\documentclass[a4paper,11pt]{amsart}
\usepackage{latexsym,bm}
\usepackage{mathrsfs,amsmath,amssymb,amsthm,enumerate}
\usepackage{color}
\usepackage{tikz}

\theoremstyle{plain}
\newtheorem{theorem}{Theorem}[section]
\newtheorem{proposition}[theorem]{Proposition}
\newtheorem{lemma}[theorem]{Lemma}
\newtheorem{corollary}[theorem]{Corollary}
\numberwithin{equation}{section}

\theoremstyle{definition}
\newtheorem{definition}[theorem]{Definition}
\newtheorem{remark}[theorem]{Remark}
\newtheorem{example}[theorem]{Example}

\newcommand{\C}{\mathbb{C}}
\newcommand{\Q}{\mathbb{Q}}
\newcommand{\R}{\mathbb{R}}
\newcommand{\Z}{\mathbb{Z}}

\newcommand{\N}{\mathbb{N}}

\newcommand{\B}{\mathcal{B}}
\newcommand{\ta}{ta}
\newcommand{\tsa}{tsa}

\def\red#1{{\textcolor{red}{#1}}}

\newcommand{\atopp}[2]{\genfrac{}{}{0pt}{}{#1}{#2}}

\DeclareMathOperator{\rank}{rank}

\begin{document}
\title[A new graph invariant arises in toric topology]{A new graph invariant arises in toric topology}

\author[S.Choi]{Suyoung Choi}
\address{Department of Mathematics, Ajou University, San 5, Woncheondong, Yeongtonggu, Suwon 443-749, Korea}
\email{schoi@ajou.ac.kr}

\author[H.Park]{Hanchul Park}
\address{Department of Mathematics, Ajou University, San 5, Woncheondong, Yeongtonggu, Suwon 443-749, Korea}
\email{hpark@ajou.ac.kr}

\thanks{
The authors were partially supported by Basic Science Research Program through the National Research Foundation of Korea(NRF) funded by the Ministry of Education, Science and Technology(NRF-2012R1A1A2044990). The first author is additionally supported by TJ Park Science Fellowship funded by POSCO TJ Park Foundation.}

\date{\today}


\begin{abstract}

    In this paper, we introduce new combinatorial invariants of any finite simple graph, which arise in toric topology. We compute the $i$-th (rational) Betti number and Euler characteristic of the real toric variety associated to a graph associahedron $P_{\B(G)}$. They can be calculated by a purely combinatorial method (in terms of graphs) and are named $a_i(G)$ and $b(G)$, respectively. To our surprise, for specific families of the graph $G$, our invariants are deeply related to well-known combinatorial sequences such as the Catalan numbers and Euler zigzag numbers.
\end{abstract}
\maketitle

\tableofcontents

\section{Introduction}

    For a finite simple graph $G$, we define a graph invariant called the \emph{signed $a$-number} of $G$, written as $sa(G)$, as follows:
    \begin{itemize}
        \item $sa(G)$ is the product of signed $a$-numbers of connected components of $G$. In particular, $sa(\varnothing)=1$.
        \item $sa(G)=0$ if $G$ has odd order.
        \item If $G$ is a connected graph of even order, then $sa(G)$ is given by minus the sum of signed $a$-numbers of all induced subgraphs of $G$ other than $G$ itself.
    \end{itemize}
    The \emph{$a$-number} of $G$, written as $a(G)$, is defined by the absolute value of $sa(G)$. The \emph{$i$-th $a$-number} of $G$, $a_i(G)$, is the sum of $a$-numbers of induced subgraphs of $G$ of order $2i$. The \emph{total $a$-number} $b(G)$ is the sum of signed $a$-numbers of every induced subgraphs of $G$. In Section~\ref{sec:anumber}, we compute these invariants for specific classes of graphs and present tables for them.

    These numerical invariants are derived from certain topological invariants of real toric manifolds, which are one of important objects in toric topology. A toric variety of complex dimension $n$ is a normal algebraic variety over the complex field $\C$ with an effective algebraic action of $(\C^*)^n$ having an open dense orbit, where $\C^*=\C\setminus\{0\}$. %A fan of real dimension $n$ is a collection of cones in $\R^n$ with the origin as the vertex satisfying certain conditions. A fundamental result of toric geometry says that there is a bijective correspondence between toric varieties of complex dimension $n$ and fans of real dimension $n$, see \cite{Fu} and \cite{Od}.
    A compact non-singular toric variety is called a toric manifold; the subset consisting of points with real coordinates is called a real toric manifold.

%    Toric manifolds include an important family of symplectic manifolds called symplectic toric manifolds. A symplectic toric manifold is defined to be a compact, connected symplectic manifold of dimension $2n$ with an effective hamiltonian action of a torus $T^n$. On the other hand,
    A simple polytope $P^n$ is called a \emph{Delzant polytope} if for each vertex $p$ of $P^n$, the outward normal vectors of the facets containing $p$ can be chosen to make up an integral basis for $\Z^n$.
    Note that the normal fan of a Delzant polytope is a complete non-singular fan and thus defines a toric manifold by the fundamental theorem of toric geometry.

%    Delzant's classification theorem \cite{De} states that there is a one-to-one correspondence between symplectic toric manifolds and Delzant polytopes.

    There is an interesting family of Delzant polytopes called nestohedra {introduced in \cite{P05}}. Let us define some {terminology}. A \emph{building set} $\mathcal{B}$ on a finite set $S$ is a collection of nonempty subsets of $S$ such that
        \begin{enumerate}
            \item $\mathcal{B}$ contains all singletons $\{i\}$, $i\in S$,
            \item if $I, J \in \mathcal{B}$ and $I\cap J\neq\varnothing$, then $I\cup J\in\mathcal{B}$.
        \end{enumerate}
     Let $\B$ be a building set on $[n+1]=\{1,\ldots,n+1\}$. For $I\subset [n+1]$, let $\Delta_I$ be the simplex given by the convex hull of points $e_i$, $i\in I$, where $e_i$ is the $i$-th coordinate vector. Then define \emph{the nestohedron} $P_\B$ as the Minkowski sum of simplices
        $$ P_\B = \sum_{I\in\B} \Delta_I. $$
    See \cite{PRW} or Section~\ref{sec:buildingset} for details. It is well-known that every nestohedron is a Delzant polytope (for example, see \cite[Proposition 7.10]{P05}). If $G$ is a graph and $\B=\B(G)$ is a building set whose elements are obtained from connected induced subgraphs of $G$, then $P_{\B(G)}$ is called a \emph{graph associahedron}. {The notion of graph associahedra was introduced in \cite{CD} motivated by \cite{DJS}.} The class of graph associahedra includes some important families of simple polytopes, such as permutohedra $Pe^{n}$, associahedra $As^{n}$ (or Stasheff polytopes), cyclohedra $Cy^{n}$ (or Bott-Taubes polytopes) and stellohedra $St^{n}$, corresponding to the complete graphs $K_{n+1}$, the path graphs $P_{n+1}$, the circle graphs $C_{n+1}$, and the star graphs $K_{1,n}$ with $n+1$ vertices respectively. Note that star graph $K_{1,n}$ is a special kind of complete bipartite graphs $K_{m,n}$.

    Since nestohedra are Delzant polytopes,
    %by Delzant's classification theorem,
    {we} have a toric manifold associated to the graph associahedron $P_{\B(G)}$ which is denoted by $M_\C(\B(G))$. Its real toric manifold is written as $M_\R(\B(G))$. In the toric manifold case, one can use the famous {results of Jurkiewicz \cite{Ju} and Danilov \cite{Da}} to compute the cohomology ring. In particular, the Betti numbers are given by the $h$-vector of $P_{\B(G)}$, which is a combinatorial invariant determined by number of faces of the polytope. So the problem to find the Betti numbers of $M_\C(\B(G))$ reduces to that of computing $h$-vectors of the graph associahedron $P_{\B(G)}$. See \cite{PRW}. In this paper, we focus on the real toric manifold $M_\R(\B(G))=:M(G)$. In this case, the theorem of Davis-{Januszkiewicz} \cite[Theorem~4.14]{DJ} tells only about $\Z_2$-coefficient version $H^*(M(G);\Z_2)$. Thus, we want to compute rational Betti numbers of $M(G)$.
%    Recently, Suciu \cite{Su1} presented an amazing method to compute the rational Betti numbers of any real toric manifold. By applying Suciu's method,
    Hereby we present the main result:

    \begin{theorem}\label{thm:mainintro}
        Let $G$ be a graph (not necessarily connected). Then the rational Betti numbers $\beta_i(M(G))$ and the Euler characteristic $\chi(M(G))$ of $M(G)$ are
        $$
            \beta_i(M(G))= a_i(G) \text{ and } \chi(M(G)) = b(G).
        $$
    \end{theorem}
    \begin{remark}\label{rem:bnumandeuler}
        By a result of Davis-Januskiewicz \cite{DJ}, the $\Z_2$-Betti numbers of $M_\R(\B(G))$ is equal to the $\Q$-Betti numbers of $M_\C(\B(G))$, which is given by the $h$-vector of $P_{\B(G)}$ as mentioned above. Since the Euler characteristic can be calculated using any coefficient field \cite[Exercise~3A.1]{H}, one concludes that $b(G)$ also can be obtained from the $h$-vector of $P_{\B(G)}$. See Remark~\ref{rem:kindreferee} for details.
    \end{remark}

    An amazing formula by Suciu and Trevisan \cite{Su1} to calculate rational Betti number of any real toric manifold is one of the key tools in the proof of Theorem~\ref{thm:mainintro}.
    As immediate consequence, we obtain the following corollary.
    \begin{corollary}
        If $G=K_{n+1}$ is a complete graph, then
        $$
            \beta_i(M(G))=a_i(K_{n+1})=\binom{n+1}{2i}A_{2i}
        $$
        and
        $$
            \chi(M(G))=b(K_{n+1})=\left\{
                         \begin{array}{ll}
                           0, & \hbox{if $n$ is odd;} \\
                           (-1)^{\frac n2}A_{n+1}, & \hbox{if $n$ is even,}
                         \end{array}
                       \right.
        $$
    where $A_k$ is the $k$-th Euler zigzag number.
    \end{corollary}
    The toric variety $M_\C(\B(K_{n+1}))$ is known as a \emph{Hessenberg variety} \cite{DeMPS} and its real version $M(K_{n+1})$ is also well-studied. In particular, its rational Betti numbers have already been computed by Henderson \cite[Corollary 1.3]{Hen} using a geometrical approach. After that, Suciu \cite{Su2} also computed it using his own method. We remark that our result can be regarded as a generalization of Suciu's.

    \begin{corollary}
        If $G=P_{n+1}$ is a path graph, then
        $$
            \beta_i(M(G))=a_i(P_{n+1})=\binom{n+1}i-\binom{n+1}{i-1}
        $$
        for $1\le i\le\lfloor\frac {n+1}2\rfloor$ and
        $$
            \chi(M(G))=b(P_{n+1})=\left\{
                         \begin{array}{ll}
                           0, & \hbox{if $n$ is odd;} \\
                           (-1)^{\frac n2}\mathcal{C}_{\frac n2}, & \hbox{if $n$ is even,}
                         \end{array}
                       \right.
        $$
        where $\mathcal{C}_k=\frac 1{k+1}\binom{2k}k$ is the $k$-th Catalan number.
    \end{corollary}
    One can find the list of combinatorial interpretations of $\mathcal{C}_n$ developed by R.~Stanley at \texttt{http://www-math.mit.edu/\textasciitilde rstan/ec/.}
    It is noted that $a_n(P_{2n}) = |b(P_{2n+1})|$ is the $n$-th Catalan number $\mathcal{C}_n$. Since $a_i(G)$ and $b(G)$ are calculated in a purely combinatorial way, this result has been included recently in Stanley's list as a new combinatorial interpretation of the Catalan numbers (see \cite[C.6C]{Stan}).

    \begin{corollary}
        If $G=C_{n+1}$ is a cycle graph, then
        $$
            \beta_i(M(G))=a_i(C_{n+1})=\left\{
                                         \begin{array}{ll}
                                           \binom{n+1}i, & \hbox{if $2i< n+1$;} \\
                                           \frac 12\binom{2i}i, & \hbox{if $2i= n+1$.}
                                         \end{array}
                                       \right.
        $$
        and
        $$
            \chi(M(G))=b(C_{n+1})=\left\{
                         \begin{array}{ll}
                           0, & \hbox{if $n$ is odd;} \\
                           (-1)^{\frac n2}\binom{n}{n/2}, & \hbox{if $n$ is even.}
                         \end{array}
                       \right.
        $$
    \end{corollary}
    \begin{corollary}
        If $G=K_{1,n}$ is a star graph, then
        $$
            \beta_i(M(G))=a_i(K_{1,n})=\binom{n}{2i-1}A_{2i-1}
        $$
        for $i\ge 1$ and
        $$
            \chi(M(G))=b(K_{1,n})=\left\{
                         \begin{array}{ll}
                           0, & \hbox{if $n$ is odd;} \\
                           (-1)^{\frac n2}A_{n}, & \hbox{if $n$ is even,}
                         \end{array}
                       \right.
        $$
    where $A_k$ is the $k$-th Euler zigzag number.
    \end{corollary}

    This paper is organized as follows. In Section~\ref{sec:anumber}, we define our graph invariants containing the signed and unsigned $a$-numbers, the $i$-th $a$-numbers, and the total $a$-numbers. Furthermore, we compute them for specific classes of graphs such as $P_n$, $C_n$, $K_n$, and $K_{1,n-1}$, and give tables for them. In Section~\ref{sec:pre}, we recall the definition of small covers and introduce the formula of Suciu-Trevisan. We also review nestohedra and graph associahedra. In Section~\ref{sec:kgeven}, we introduce the simplicial complex $K_G^{\mathrm{even}}$ whose topology is essential to the computation. We also introduce a subdivision of $K_G^{\mathrm{even}}$ that is shellable, which implies $K_G^{\mathrm{even}}$ is homotopy equivalent to a wedge sum of spheres of the same dimension. Finally, in Section~\ref{sec:betti}, we prove Theorem~\ref{thm:mainintro}.

\section{$a$-numbers: definition and examples}\label{sec:anumber}
    Throughout this paper, every graph is assumed to be finite, undirected, and simple. We start by defining our invariant, called the \emph{$a$-number}. For a graph $G$, the set of vertices and edges are denoted by $V(G)$ and $E(G)$, respectively.

    \begin{definition}
        Let $G$ be a graph. The \emph{signed $a$-number} of $G$, or $sa(G)$, is defined recursively by the following conditions:
        \begin{itemize}
            \item $sa(G)=\prod_{i=1}^\ell sa(G_i)$ if $G_1,\ldots, G_\ell$ are components of $G$. In particular, $sa(\varnothing) = 1$.
            \item If $G$ is connected, then:
        \begin{equation} \label{rec}
            sa(G) = \left\{
                      \begin{array}{ll} \displaystyle
                        -\sum_{I\subsetneq V(G)}sa(G|_I), & \hbox{if $G$ has even order;} \\
                        0, & \hbox{otherwise,}
                      \end{array}
                    \right.
        \end{equation}
        where $G|_I$ is the full subgraph of $G$ induced by $I$, i.e., $V(G|_I) = I$ and $E(G|_I) = \{ \{v,w\} \in E(G) \mid v, w \in I\}$.
        \end{itemize}
%
%        \begin{itemize}
%            \item If $G$ has an odd number of vertices, $sa(G)=0$ .
%%            \item $sa(G)$ is the product of signed $a$-numbers of components of $G$.
%            \item Otherwise, then
%               \begin{equation}
%                sa(G)=-\sum_{I\subsetneq V(G)}sa(G|_I). \label{rec}
%               \end{equation}
%             %$sa(G)$ is given by minus the sum of signed $a$-numbers of all induced subgraphs of $G$ other than $G$ itself.
%        \end{itemize}

        The \emph{$a$-number} or \emph{unsigned $a$-number} of $G$, denoted by $a(G)$, is the absolute value of $sa(G)$. The \emph{$i$-th $a$-number} of $G$ or $a_i(G)$ is defined by the sum
        $$
            a_i(G):=\sum_{\atopp{I\subseteq V(G)}{|I|=2i}}a(G|_I).
        $$
        Note that $a_1(G)$ is the number of edges of $G$.

        The \emph{total $a$-number} of $G$, or $b(G)$, is the whole sum of signed $a$-numbers of all induced subgraphs, that is
        $$
            b(G):=\sum_{I\subseteq V(G)}sa(G|_I).
        $$
    \end{definition}
    \begin{remark}
        Even though it seems nontrivial from the definition, the relation $$sa(G)=(-1)^{|\frac{V(G)}2|}a(G)$$
        holds. As we shall see in the proof of Theorem~\ref{thm:mainintro}, this is an obvious fact from topological viewpoint. Assuming this relation, it is easy to see that
            $$
                b(G)=\sum_{i=0}^{\lfloor \frac{V(G)}2 \rfloor} (-1)^i a_i(G).
            $$
    \end{remark}
    \begin{remark}\label{rem:kindreferee}
        As we have seen in Remark~\ref{rem:bnumandeuler}, $b(G)$ can be computed from the $h$-vector of $P_{\B(G)}$. More precisely, when $G$ is a graph with $2k+1$ vertices, the following
        \[
            b(G) = f(G,-2) = h(G,-1) = (-1)^k\times\text{coeff of }t^k\text { in }\gamma(G,t)
        \]
        holds where $f(G,t),\,h(G,t),\,\gamma(G,t)$ denote the $f$-, $h$-, $\gamma$-polynomials of the polytope $P_{\B(G)}$. This can be proven by checking that $b(G)$ satisfies the recurrence relations of \cite[Theorem~7.11]{P05}.
    \end{remark}
    \begin{remark}
        If $G$ is a connected graph with $2n$ vertices, $n\ge 1$, then $a_n(G)=a(G)$ and $b(G)=0$. Therefore, if $b(G)$ is nonzero, then $G$ has odd order.
    \end{remark}

    The rest of this section is devoted to calculating $a$-numbers of some important examples of graphs, such as path graphs, complete graphs, star graphs, and cycle graphs.

    \begin{theorem}\label{thm:anumberforpath}
        Let $G=P_n$ be the path graph with $n$ vertices. Then
        $$
            sa(P_{2n})=(-1)^n\frac1{n+1}\binom{2n}{n}
        $$
        is the $n$-th Catalan number up to sign. More generally the following holds:
        $$
            a_i(P_n)=\binom{n}{i} - \binom{n}{i-1}
        $$
        for $1\le i\le\lfloor\frac n2\rfloor$.
%Note that $a(P_{2n})=a_n(P_{2n})$.
    \end{theorem}
    \begin{proof}
        First, we verify the first formula. Put $G=P_{2n}$ and assume that $V(G)=[2n]=\{1,\ldots,2n\}$ and the edges are of the form $\{k,k+1\}$, $1\le k\le 2n-1$. To compute $sa(P_{2n})$ we must check out every induced subgraph of $G$ whose signed $a$-number is nonzero. Pick two vertices of $G$, named $v$ and $w$. We can assume that $1\le v < w \le 2n$. Let $I\subseteq [2n]$ be a subset of $[2n]$ and suppose that $v$ and $w$ are the first two vertices of $G$ which are not contained in $I$, that is, $\{1,2,\ldots, v-1, v+1, v+2,\ldots, w-2, w-1\}\subset I$ and $v,w\notin I$.

        Now, consider the sum of signed $a$-numbers of $G|_I$ for $I$ satisfying the above conditions, denoted by $S(v,w)$. Observe that $S(v,w)=sa(P_{v-1})\cdot sa(P_{w-v-1})\cdot b(P_{2n-w})$. We only need to consider the cases $v$ is odd and $w$ is even. Then $2n-w$ is even and $b(P_{2n-w})$ is zero unless $w=2n$. In conclusion, summing $S(v,w)$ whenever $v$ is odd and $w=2n$ gives us the result
        $$
            -sa(P_{2n})=sa(P_0)sa(P_{2n-2})+sa(P_2)sa(P_{2n-4})+\dotsb+sa(P_{2n-2})sa(P_0),
        $$
        which is the famous recurrence relation for the Catalan number (except the signs). Therefore the first part of the theorem is proven.

        For the second part, assume that $V(P_n)=[n]=\{1,\ldots,n\}$ and the edges are of the form $\{k,k+1\}$, $1\le k\le n-1$.
        Suppose $X=\{x_1,\ldots,x_i\}$ be a subset of $[n]$. Define $\bar{X}=\{x_1,\ldots,x_i,x'_1,\ldots,x'_i\}\subset \Z$ be the unique set satisfying the following conditions:
            \begin{enumerate}
                \item $|\bar{X}|=2i$.
                \item If $k$ is an integer between $x_a$ and $x'_a$, then $k\in\bar{X}$.
                \item $x'_a<x_a$ for all $1\le a \le i$.
            \end{enumerate}

        Let $A_i$ be the set of subsets of $[n]$ with cardinality $i$. Then $|A_i|=\binom{n}{i}$. Let $B_i$ be the set of $X\in A_i$ such that the minimum of $\bar{X}$ is non-positive. We claim that $|B_i|= \binom n{i-1}$. To prove it, we give a one-to-one correspondence from $B_i$ to $A_{i-1}$. Suppose $X=\{x_1,\ldots,x_i\}\in A_i$ and $x_1<\dotsb <x_i$. Then $X\in B_i$ if and only if $x_j\le 2j-1$ for some $j$. Actually the equality holds since if $x_j < 2j-1$, then $x_{j-1} \le 2j-3=2(j-1)-1$ and we could assume $j$ was minimal. So, let $j$ be the minimal index such that $x_j=2j-1$. Now define $f:B_i\to A_{i-1}$ by $f(X)=(X\setminus [2j-1])\cup([2j-1]\setminus X)$. Now, we consider the inverse of $f$, say $g$. Suppose $Y\in A_{i-1}$. If $1\notin Y$, then $g(Y)= Y\cup\{1\}$. If $1\in Y$ and $Y$ does not contain any of 2 or 3, then $g(Y)=Y\cup\{2,3\}\setminus\{1\}$. In general, there is a $j$ such that $|Y\cap[2j-1]| = j-1$. Take the minimal $j$ and define $g(Y):=(Y\setminus [2j-1])\cup([2j-1]\setminus Y)$. It is an easy exercise to show that $g$ is well-defined and $g=f^{-1}$. Note that if $|Y\cap[2j-1]| = j-1$ and $2j-1\in Y\cap[2j-1]$, then $Y\cap [2j-2]$ has $j-2$ elements and $j$ cannot be minimal no matter whether $Y$ contains $2j-2$ or not.

        Now we consider the elements of $A_i\setminus B_i$. For any $I\in A_i\setminus B_i$, the induced subgraph $P_n|_{\bar I}$ has no component of odd order. We claim that $a(P_n|_{\bar I})$ is equal to the number of $J$'s such that $\bar J = \bar I$. It is enough to check it for the case that $G_{\bar I}$ is connected. That is, we should count the number of $J$'s such that $\bar I=\bar J$ when $\bar I=[2k]$ for some $k$. But it is exactly the $k$-th Catalan number. To show it, for example, consider the function $t: \bar I\to \{(,)\}$ such that
        $$t(x)=\left\{
          \begin{array}{ll}
            (, & \hbox{if $x\notin J$;} \\
            ), & \hbox{if $x\in J$.}
          \end{array}
        \right.$$
        Recall that the $k$-th Catalan number counts the number of correct expressions of $k$ pairs of parentheses. Since we already have shown that $a(P_{2k})$ is the $k$-th Catalan number, the claim is proven, completing the proof.
    \end{proof}
    \begin{table}
    \centering
    \begin{tabular}{|c|cccccc|}
      \hline
      $a_i(P_{n})$ & $i=0$ & 1 & 2 & 3 & 4 & 5  \\ \hline
      $n=0$ & 1  & & & & &  \\
      1 & 1 & & & & &   \\
      $2$ & 1 & 1 &  &  &  &    \\
      3 & 1 & 2 &  &  &  &    \\
      4 & 1 & 3 & 2 &  &  &    \\
      5 & 1 & 4 & 5 &  &  &    \\
      6 & 1 & 5 & 9 & 5 &  &    \\
      7 & 1 & 6 & 14 & 14 &  &     \\
      8 & 1 & 7 & 20 & 28 & 14 &    \\
      9 & 1 & 8 & 27 & 48 & 42 &     \\
      10 & 1 & 9 & 35 & 75 & 90 & 42    \\
      \hline
    \end{tabular}
    \medskip
    \caption{Values of $a_i(P_n)$ make up a Catalan triangle.}\label{pathtable}
    \end{table}

    The bisequence $a_i(P_n)$ turns out to be the famous \emph{Catalan triangle}, A008315 of \cite{O}. See Table \ref{pathtable}.

%    The next example is the cycle graph.

    \begin{theorem}\label{thm:anumberforcircle}
        Let $G=C_n$ be the cycle graph with $n$ vertices. Then
        $$
            sa(C_{2n})=(-1)^n\binom{2n-1}{n-1} = (-1)^n \frac12\binom{2n}n,
        $$
        i.e., it is the half of the $n$-th central binomial coefficient up to sign. Moreover,
        $$
            a_i(C_n)=\binom ni
        $$
        if $i<\frac n2$.
    \end{theorem}

    \begin{proof}
        As in the proof of the previous theorem, we may show that $a_i(C_n)=\binom ni$ if $i<\frac n2$. Indeed, noticing that $G|_{\bar X}$ is a proper subgraph of $G$, everything works the same way, except that $x'_a$ should be in counterclockwise (or clockwise) direction when seen from $x_a$ in $G|_{\bar X}$.
%        First, it is an analogue of proof of the previous theorem to show $a_i(C_n)=\binom ni$ if $i<\frac n2$. Noticing $G|_{\bar X}$ is a proper subgraph of $G$, everything goes the same but $x'_a$ should be in counterclockwise (or clockwise) direction seen from $x_a$ in $G|_{\bar X}$.

        So we are left with computing $sa(C_{2n})$. We assume the vertices are named $1, 2, \ldots, 2n$, and the edges connect $j$ and $j+1$ mod $2n$. First, if we do not choose the vertex 1, consider this condition: we do not choose 1 and $j$ and we choose every vertices $2, 3, \ldots, j-1$. We freely choose or not the vertices $j+1, j+2, \ldots, 2n$. Each choice gives a subset $I$ of the vertex set. For fixed $j$, we can compute the sum of $a$-numbers on $C_{2n}|_I$ for $I$ satisfying above condition. Similarly to the previous theorem, only the case $j=2n$ is nontrivial. If we do choose 1, we have two vertices $a$ and $b$ such that $1<a<b\le 2n$ and $1,2,\ldots,a-1$ and $b+1,b+2,\ldots, 2n$ are chosen and $a$ and $b$ are not chosen. Then we have nontrivial contributions only if $a$ and $b$ are adjacent, i.e., $b=a+1$, $2\le a\le 2n-1$. Therefore we have
        \begin{align*}
            sa(C_{2n})&=-(2n-1)sa(P_{2n-2})=-(2n-1)\cdot (-1)^{n-1}\frac{1}n\binom{2n-2}{n-1}\\
                    &=(-1)^n\binom{2n-1}n.
        \end{align*}
    \end{proof}
    The bisequence $a_i(C_n)$ makes `half' of Pascal's triangle or A008314 of \cite{O}. See Table \ref{cycletable}.

    \begin{table}
    \centering
    \begin{tabular}{|c|cccccc|}
      \hline
      $a_i(C_{n})$ & $i=0$ & 1 & 2 & 3 & 4 & 5  \\ \hline
      $n=0$ & 1  & & & & &  \\
      1 & 1 & & & & &   \\
      $2$ & 1 & 1 &  &  &  &    \\
      3 & 1 & 3 &  &  &  &    \\
      4 & 1 & 4 & 3 &  &  &    \\
      5 & 1 & 5 & 10 &  &  &    \\
      6 & 1 & 6 & 15 & 10 &  &    \\
      7 & 1 & 7 & 21 & 35 &  &     \\
      8 & 1 & 8 & 28 & 56 & 35 &    \\
      9 & 1 & 9 & 36 & 84 & 126 &     \\
      10 & 1 & 10 & 45 & 120 & 210 & 126    \\
      \hline
    \end{tabular}
    \medskip
    \caption{Values of $a_i(C_n)$.}\label{cycletable}
    \end{table}

%    We introduce a sequence of positive integers.
    \begin{definition}
        Let $\{A_n\}$ be the sequence given by
        $$
            \sec x + \tan x = \sum_{n=0}^\infty A_n \frac{x^n}{n!}.
        $$
        The numbers $A_n$ are known as \emph{Euler zigzag numbers}. The numbers $A_{2i}$ with even indices are called \emph{secant numbers} and $A_{2i+1}$ with odd ones are \emph{tangent numbers}.
    \end{definition}
   A permutation $\sigma$ of $[n]$ is called \emph{alternating} if $\sigma(2i-1)<\sigma(2i)$ and $\sigma(2j)>\sigma(2j+1)$ for all $i$ and $j$. In fact, the Euler zigzag number $A_n$ is the number of \emph{alternating permutations} of $[n]$.

    \begin{theorem}\label{theoremsecant}
        Let $G=K_n$ be the complete graph with $n$ vertices. Then
        $$
            sa(K_{2n})=(-1)^n A_{2n},
        $$
        where $A_{2n}$ is the secant number. In general,
        $$
            a_i(K_n)=\binom n{2i} A_{2i}.
        $$
    \end{theorem}
    \begin{proof}
        We have a recurrence relation
        \begin{equation} \label{complete_graph}
            \binom{2n}0sa(K_0)+\binom{2n}2sa(K_2)+\ldots+\binom{2n}{2n}sa(K_{2n})=0
        \end{equation}
        if $n\ge 1$. Let us write $sa(K_{2i})=:X_{2i}$ and $F(x)$ be the formal series
        $$
            F(x)=\sum_{i=0}^{\infty}|X_{2i}|\frac{x^{2i}}{(2i)!}.
        $$
        Then
        $$
            F(x)\cdot \cos x= \left(\sum_{i=0}^{\infty}(-1)^i X_{2i}\frac{x^{2i}}{(2i)!}\right)
            \left(\sum_{j=0}^{\infty}(-1)^j \frac{x^{2j}}{(2j)!}\right)
        $$
        is equal to 1 using the recurrence relation above and the fact $X_0=1$. Hence $F(x)=\sec x$ and it is done.
    \end{proof}

    \begin{theorem} \label{theoremtangent}
        Let $G=K_{1,n-1}$ be the star graph with $n$ vertices for $n\ge 1$. Then
        $$
            sa(K_{1,2n-1})=(-1)^n A_{2n-1},
        $$
        where $A_{2n-1}$ is the tangent number. Moreover,
        $$
            a_i(K_{1,n-1})=\binom {n-1}{2i-1} A_{2i-1}
        $$
        for $i\ge 1$.
    \end{theorem}
    \begin{proof}
        The proof is almost identical to that of Theorem \ref{theoremsecant}. In this case the recurrence relation is
        $$
            sa(\varnothing)+ \sum_{j=1}^{n} \binom{2n-1}{2j-1} sa(K_{1,2j-1}) = 0,
        $$
        if $n\ge 1$. Write $Y_{2i-1}=sa(K_{1,2i-1})$ and let $F(x)$ be the formal series
        $$
            F(x)=\sum_{i=1}^{\infty}(-1)^i Y_{2i-1} \frac{x^{2i-1}}{(2i-1)!}.
        $$
        It is enough to show that $F(x)=\tan x$. Just calculate $F(x)\cdot \cos x$ and check that it becomes $\sin x$.
    \end{proof}

    \begin{table}
    \centering
    \begin{tabular}{|c|ccccc|}
      \hline
      $a_i(K_{n})$ & $i=0$ & 1 & 2 & 3 & 4  \\ \hline
      $n=0$ & 1  & & & &   \\
      1 & 1 & & & &    \\
      $2$ & 1 & 1 &  &  &      \\
      3 & 1 & 3 &  &  &      \\
      4 & 1 & 6 & 5  &    &  \\
      5 & 1 & 10 & 25 &  &      \\
      6 & 1 & 15 & 75 & 61  &      \\
      7 & 1 & 21 & 175 & 427 &       \\
      8 & 1 & 28 & 350 & 1708 & 1385      \\
%      9 & 1 & 9 & 36 & 84 & 126 &     \\
%      10 & 1 & 10 & 45 & 120 & 210 & 126    \\
      \hline
    \end{tabular}
    \caption{Values of $a_i(K_n)$.}\label{completetable}
    \end{table}

    \begin{table}
    \centering
    \begin{tabular}{|c|ccccc|}
      \hline
      $a_i(K_{1,n-1})$ & $i=0$ & 1 & 2 & 3 & 4  \\ \hline
      $n=0$ & 1  & & & &  \\
      1 & 1 & & & &   \\
      $2$ & 1 & 1 &  &  &     \\
      3 & 1 & 2 &  &  &      \\
      4 & 1 & 3 & 2  &    &  \\
      5 & 1 & 4 & 8 &  &      \\
      6 & 1 & 5 & 20 & 16  &      \\
      7 & 1 & 6 & 40 & 96 &       \\
      8 & 1 & 7 & 70 & 336 & 272      \\
%      9 & 1 & 9 & 36 & 84 & 126 &     \\
%      10 & 1 & 10 & 45 & 120 & 210 & 126    \\
      \hline
    \end{tabular}
    \caption{Values of $a_i(K_{1,n-1})$.}\label{startable}
    \end{table}

    Table~\ref{completetable} and Table~\ref{startable} describe $i$-th $a$-numbers of $K_n$ and $K_{1,n-1}$, respectively. Especially, Table~\ref{completetable} is the unsigned version of A153641 of \cite{O}, which is nonzero coefficients of the Swiss-Knife polynomials which can be used to compute secant numbers, tangent numbers or Bernoulli numbers.

    \begin{remark}
        There is a combinatorial way to prove Theorem~\ref{theoremsecant} using \eqref{complete_graph}. See the second proof of \cite[Theorem 1.1]{St} for example. Theorem~\ref{theoremtangent} also can be proved in similar way.
    \end{remark}

    \begin{remark}
        It {can be shown} that the total $a$-numbers of our examples are given by
\begin{align*}
            b(P_{2n+1})&=(-1)^n \frac 1{n+1}\binom{2n}n,  \\
            b(C_{2n+1})&=(-1)^n \binom{2n}n, \\
            b(K_{2n+1})&=(-1)^n A_{2n+1}, \quad \text{and} \\
            b(K_{1,2n})&=(-1)^n A_{2n}.
\end{align*}
        These identities certainly can be proven using $a_i(G)$. On the other hand, they indeed can be deduced by $h$-vectors of $P_{2n+1},C_{2n+1},K_{2n+1}$, and $K_{1,2n}$, which are excellently described in \cite[Section~11]{PRW}.
        %{For the sake of brevity, we will omit the proofs of these identities.}
%In particular, the complete graphs and the star graphs seem to have some kind of `dual' relationship.
    \end{remark}

\section{Preliminaries}\label{sec:pre}

    \subsection{Real toric manifolds and their rational homology: Suciu-Trevisan formula}
    Returning to topology, we need some notions from toric geometry and toric topology. A small cover, introduced by Davis-{Januszkiewicz} in \cite{DJ}, is a topological analogue of a real toric manifold. An $n$-dimensional closed smooth manifold $M$ is called a \emph{small cover over $P$} if it has a group action of $\Z_2^n$ locally isomorphic to the standard representation of $\Z_2^n$ on $\R^n$, and the orbit space $M/\Z_2^n$ can be identified with a simple polytope $P$ of dimension $n$. Let $P$ be a simple polytope with the facet set $\mathcal{F}$. Associated to a small cover over $P$, there is a homomorphism $\lambda \colon \mathcal{F} \to \Z_2^n$, where $\lambda$ specifies an isotropy subgroup for each facet. We call it a \emph{characteristic function} of the small cover.

    Suciu {and Trevisan \cite{Su1} have} established a formula to compute the rational homology of a small cover as following.
    Let $P$ be a simple polytope of dimension $n$ and $M$ a small cover over $P$ with the characteristic function $\lambda$.
    Let $\mathcal{F}=\{F_1,\ldots,F_m\}$ be the set of facets of $P$.
    Then the characteristic function $\lambda \colon \mathcal{F} \to \Z_2 ^n$ can be regarded as a $\Z_2$-matrix\ of size $n\times m$, called the \emph{characteristic matrix}.
    For each subset $S$ of $[n]=\{1,\ldots,n\}$, write $\lambda_S=\sum_{i\in S}\lambda_i$, where $\lambda_i$ is the $i$-th row of $\lambda$.
    For such $S$ we define $P_S$ be the union of facets $F_j$ such that the $j$-th entry of $\lambda_S$ is nonzero.

    \begin{theorem}\label{formula}\cite{Trev, Su1}
        Let $M$ be a small cover over a simple polytope $P$ of dimension $n$. Then the (rational) Betti number of $M$ is given by
        $$
        \beta_i (M) = \sum_{S\subseteq [n]} \rank_{\Q} \tilde H_{i-1}(P_S;\Q).
        $$
    \end{theorem}
    We remark that every Delzant polytope $P$ corresponds a real toric manifold which is also a small cover over $P$, hence Theorem~\ref{formula} is applicable. The characteristic function is determined by the primitive outward normal vector to each facet of $P$.

    \subsection{Building sets, nestohedra, {and graph associahedra}}\label{sec:buildingset}
    From now on, we talk about the motivation of defining $a$-numbers for graphs. As we mentioned in Introduction, $a$-numbers become the rational Betti numbers of real toric manifolds arising from the specific polytope associated to a simple graph. In this section, we briefly review about the graph associahedra. See \cite{P05} for details.

    \begin{definition}
    A \emph{building set} $\mathcal{B}$ on a finite set $S$ is a collection of nonempty subsets of $S$ such that
        \begin{enumerate}
            \item $\mathcal{B}$ contains all singletons $\{i\}$, $i\in S$,
            \item if $I, J \in \mathcal{B}$ and $I\cap J\neq\varnothing$, then $I\cup J\in\mathcal{B}$.
        \end{enumerate}
        If $\B$ contains the whole set $S$, then $\B$ is called \emph{connected}.
    \end{definition}

    \begin{example}
        Let $G$ {be} a finite (simple) graph with the vertex set $S$. The \emph{graphical building set} $\mathcal{B}(G)$ is defined by
        $$
        \mathcal{B}(G)=\{J\subseteq S \mid G|_J\text{ is connected}\},
        $$
    where $G|_J$ is the induced subgraph of $G$ on $J$. It is obvious that $\mathcal{B}(G)$ is a building set. A building set $\B(G)$ is connected if and only if $G$ is a connected graph.
    \end{example}

    For a building set $\B$, we can assign a simple polytope called a \emph{nestohedron}:
    \begin{definition}
        Let $\B$ be a building set on $[n+1]=\{1,\ldots,n+1\}$. For $I\subset [n+1]$, let $\Delta_I$ be the simplex given by the convex hull of points $e_i$, $i\in I$, where $e_i$ is the $i$-th coordinate vector. Then define \emph{the nestohedron} $P_\B$ as the Minkowski sum of simplices
        $$ P_\B = \sum_{I\in\B} \Delta_I. $$
        If $\B=\B(G)$ is a graphical building set, $P_{\B(G)}$ is called a \emph{graph associahedron}.
    \end{definition}

    If $\B$ is not connected, then the nestohedron $P_\B$ is simply a Cartesian product of the nestohedra corresponding to the maximal elements in $\B$. See \cite[Remark 6.7]{PRW} for details.
    Hence, in this paper, we deal with only connected building sets.

%    The dual simplicial complex of the boundary complex of $P_\B$ has the structure known to be the nested set complex defined below.
    \begin{definition}
        For a connected building set $\B$ on $[n+1]$, a subset $N\subseteq \B\setminus\{[n+1]\}$ is called a \emph{nested set} if the following holds:
        \begin{enumerate}[{(N}1)]
            \item For any $I, J\in N$ one has either $I\subseteq J$, $J \subseteq I$, or $I$ and $J$ are disjoint.\label{n1}
            \item For any collection of $k\geq 2$ disjoint subsets $J_1, \ldots, J_k \in N$, their union $J_1\cup \cdots \cup J_k$ is not in $\B$.\label{n2}
        \end{enumerate}
        The nested set complex $\Delta_\B$ is defined to be the set of all nested sets for $\B$.
    \end{definition}

    We note that $\Delta_\B$ is a simplicial complex.

    \begin{theorem}\cite[Theorem 7.4]{P05}\label{nested}
        Let $\B$ be a connected building set on $[n+1]$. Then the nestohedron $P_\B$ is a simple polytope of dimension $n$ and its dual simplicial complex is isomorphic to the nested set complex $\Delta_\B$.
    \end{theorem}

    Let $\B$ be a building set on $[n+1]$, and $\Delta^n$ be an $n$-simplex. Let $G_1, \ldots, G_{n+1}$ be the facets of $\Delta^n$. Then, each face of $\Delta^n$ can be uniquely expressed by $G_I = \cap_{i\in I} G_i$ for some $I\subset[n+1]$.
%Then each element $I\in \B$ is a face of the simplex $\Delta^n$ and t
The nestohedron $P_\B$ can be thought as the simplex $\Delta^n$ whose faces $G_I$ ($I\in\B$) are ``cut off'' in the following way:

    \begin{proposition}\cite[Theorem 6.1]{Z} \label{prop:obtain_P_B}
    Let $\B$ {be} a building set on $[n+1]$. Let $\varepsilon$ be a sequence of positive numbers $\varepsilon_1\ll\varepsilon_2\ll \ldots \ll \varepsilon_n\ll \varepsilon_{n+1}$. For each $I\in\B\setminus \{[n+1]\}$, assign a half-space
    $$
    A_I=\left\{(x_1,\ldots,x_{n+1})\in \R^{n+1}\middle| \sum_{i\in I}x_i \geq \varepsilon_{|I|}\right\}
    $$
    and for $I=[n+1]$, define $A_{[n+1]}$ be the hyperplane $x_1+\ldots +x_{n+1} = \varepsilon_{n+1}$. Let $P_\varepsilon$ be the intersection $\bigcap_{I\in\B} A_I$. Then one can choose $\varepsilon$ so that $P_\varepsilon=P_\B$ whose facets are given by $F_I=\partial A_I\cap P_\varepsilon$ for each $I\in\B\setminus\{[n+1]\}$. Furthermore, $P_\B$ is a Delzant polytope, i.e., the outward normal vector $\lambda(F)$ of each facet $F$ forms a basis at each vertex of $P_\B$.
    \end{proposition}

    This viewpoint would help one see the nestohedron visually and intuitively than the Minkowski sum method does. The important point is that $P_\B$ is Delzant and hence we obtain the \emph{associated toric manifold} $M_\C(\B)=M_\C(P_\B)$ and the \emph{associated real toric manifold} $M_\R(\B)$, which is the real locus of $M_\C(\B)$. From now on, our focus will be on $M_\R(\B)$.
%so occasionally it would be written by $M(\B)$ or just $M$.
In graphical case the notation $M(G):=M_\R(\B(G))$ will be also used.
\begin{example}
  Let $\B = \{\{1\}, \{2\},\{3\}, \{1,2\}, \{2,3\}, \{1,2,3\}\}$ (simply, $\B=\{1,2,3,12,23,123\}$). Since each element of $\B$ other than $123$ indicates a facet, $P_\B$ is a pentagon. Explicit geometric information obtained as in Proposition~\ref{prop:obtain_P_B} is illustrated in Figure~\ref{fig:nesto}.
  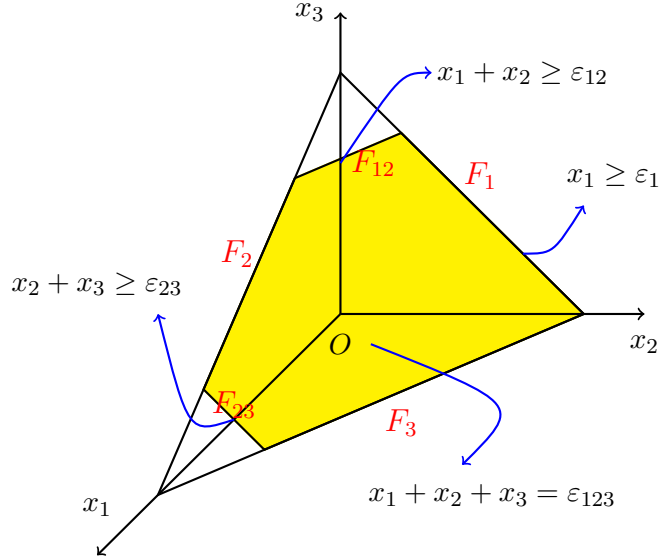
\begin{figure}[h]
    \begin{center}
 \begin{tikzpicture}[thick, scale=0.8]
    \filldraw[color=yellow](1,3)--(4,0)--(-1.25,-2.25)--(-2.25,-1.25)--(-0.75,2.25)--cycle;
    \draw (1,3)--(4,0)--(-1.25,-2.25)--(-2.25,-1.25)--(-0.75,2.25)--cycle;
    \draw (-3,-3)--(4,0)--(0,4)--cycle;
    \draw (0,-0.5) node{$O$} (5,-0.5) node{$x_2$}  (-0.5,5) node{$x_3$} (-4, -3.2) node{$x_1$} ;
    \draw[->] (0,0)--(5,0);
    \draw[->] (0,0)--(0,5);
    \draw[->] (0,0)--(-4,-4);
    \draw (2.3,2.3)node{\red{\large{$F_1$}}}
          (0, 2.5) node[right] {\red{\large{$F_{12}$}}}
          (-1.7, 1) node {\red{\large{$F_2$}}}
          (-1.75, -1.5) node {\red{\large{$F_{23}$}}}
          (1, -1.8) node{\red{\large{$F_3$}}};
    \draw[->, color=blue] (3,1)..controls (3.5, 1)..(4, 1.8);
    \draw[->, color=blue] (0.5,-0.5)..controls (3, -1.5)..(2, -2.5);
    \draw[->, color=blue] (-1.75,-1.75)..controls (-2.5, -2)..(-3, 0);
    \draw[->, color=blue] (0, 2.5)..controls (1, 4)..(1.5, 4);
    \draw (4.5, 2.3)node {$x_1 \geq \varepsilon_1$}
          (2.5, -3) node {$x_1 + x_2 + x_3 = \varepsilon_{123}$}
          (-4, 0.5) node {$x_2 + x_3 \geq \varepsilon_{23}$}
          (3, 4)node{$x_1 +x_2 \geq \varepsilon_{12}$};
 \end{tikzpicture}
    \end{center}
    \caption{An example of a (geometric) nestohedron.}\label{fig:nesto}
  \end{figure}
\end{example}

\begin{example} \label{example:associahedron}
   Let $G$ be a path graph $P_4$ with $4$ vertices.
%\psset{unit=25pt}
%\begin{center}
%\begin{pspicture}(0,0)(8,3)
%\cnode(3,1.5){2pt}{A}\uput[90](3,1.5){1}
%\cnode(4.5,1.5){2pt}{B}\uput[90](4.5,1.5){2}
%\cnode(6,1.5){2pt}{C}\uput[90](6,1.5){3}
%\cnode(7.5,1.5){2pt}{D}\uput[90](7.5,1.5){4}
%\ncline{A}{B}
%\ncline{B}{C}
%\ncline{C}{D}
%\ncline{D}{E}
%\rput(1.5, 1.5){$G=$}
%\end{pspicture}
%\end{center}
    Then, $\B(G)=\{1,2,3,4,12,23,34,123,234,1234\}$, and $P_{\B(G)}$ can be obtained as in Figure~\ref{fig:nestocut}.
   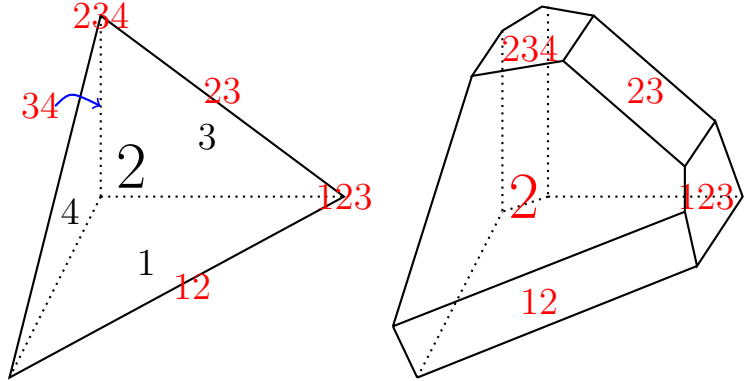
\begin{figure}[h]
        \begin{center}
\begin{tikzpicture}[thick, scale=0.4]
    \draw (8,0)--(0,6)--(-3,-6)--cycle;
    \draw[dotted](0,0)--(8,0) (0,0)--(0,6) (0,0)--(-3,-6);
    \draw (1,1) node {\Huge{$2$}}
          (1.5,-2.2) node {\Large $1$}
          (3.5,2) node {\Large $3$}
          (-1,-.5) node {\Large $4$}
          (4,3.5) node {\Large{\red{$23$}}}
          (3,-3) node {\Large{\red{$12$}}};
    \draw[color=blue, ->] (-1.5,3)..controls (-1, 3.5)..(0,3);
    \draw (-2,3) node {\Large{\red{$34$}}}
          (0,6) node {\Large{\red{$234$}}}
          (8,0) node {\Large{\red{$123$}}};
\end{tikzpicture}
\begin{tikzpicture}[thick, scale=0.4]
    \draw (0,0) node {\Huge{\red{$2$}}}
          (4,3.5) node {\Large{\red{$23$}}}
          (0.5,-3.5) node {\Large{\red{$12$}}}
          (0.2,4.9) node {\Large{\red{$234$}}}
          (6,0) node {\Large{\red{$123$}}};
    \draw (5.3,1)--(5.3,-0.5)--(5.7,-2.3)--(7.2,0)--(6.3,2.5)--cycle;
    \draw (1.3,4.5)--(2.3,6)--(0.6,6.3)--(-0.7,5.5)--(-1.7,4)--cycle;
    \draw (6.3,2.5)--(2.3,6)
          (5.3,1)--(1.3,4.5)
          (5.3,-0.5)--(-4.3,-4.3)
          (5.7,-2.3)--(-3.5,-6)
          (-1.7,4)--(-4.3,-4.3)
          (-3.5,-6)--(-4.3,-4.3);
    \draw[dotted] (-3.5,-6)--(-0.7,-0.5)
          (0.8,0)--(-0.7,-0.5)
          (0.8,6.3)--(0.8,0)
          (-0.7,5.5)--(-0.7,-0.5)
          (0.8,0)--(7.2,0);
\end{tikzpicture}
    \end{center}
    \caption{A 3-simplex and a graph associahedron, before and after ``cutting''}\label{fig:nestocut}
    \end{figure}
\end{example}

\section{{The simplicial complex $K_G^{\mathrm{even}}$} }\label{sec:kgeven}

    Let $\B$ be a building set on $[n+1]$ and $P_\B$ be the corresponding nestohedron. Let us see how to compute the outward normal vectors of the Delzant polytope $P_\B = P_\varepsilon$, where $P_\varepsilon$ is the polytope in Proposition~\ref{prop:obtain_P_B}. Let $\mathcal{F}$ be the set of facets of $P_\B$. By Theorem \ref{nested},  $\mathcal{F}$ is indexed by $\B\setminus \{[n+1]\}$ and any facet of $P_\B=P_\varepsilon$ is of the form $F_I = \partial A_I\cap P_\varepsilon$ for some $I\in\B\setminus \{[n+1]\}$. Denote the (integral and primitive) outward normal vector to $F_I$ by $\lambda(F_I)$. Note that $P_\varepsilon$ is embedded in the hyperplane $A_{[n+1]}\subseteq \R^{n+1}$. Let $\pi :\R^{n+1} \to \R^n$ be the projection on the first $n$ coordinates, i.e., $\pi(x_1,\ldots,x_n,x_{n+1})=(x_1,\ldots, x_n)$. The map $\pi$ sends $A_{[n+1]}$ onto $\R^n$, assigning it a coordinate. In that coordinate, one checks that the outward normal vector is given by
    \begin{equation}\label{eqn:lambdaofnesto}
    \lambda(F_I)=\sum_{i\in I} v_i,
    \end{equation}
    where $v_i=-e_i$, $1\leq i \leq n$, and $v_{n+1}=e_1+\ldots+e_n$ and $e_i$ is the $i$-th coordinate vector of $\R^n$.

    As a small cover, the characteristic function of $M_\R(\B)$ is given by $\lambda$ modulo $2$, where $\lambda$ is given by \eqref{eqn:lambdaofnesto}. We abuse the notation $\lambda$ for the modulo $2$ reduction. The characteristic matrix for $\lambda$, again written as $\lambda=(\lambda_{iI})$, is an $n\times (|\B|-1)$ matrix as a $\Z_2$-matrix. By \eqref{eqn:lambdaofnesto}, the entry $\lambda_{iI}$ can be computed as
    $$
    \lambda_{iI}=\left\{
        \begin{array}{lll}
          1,&i\in I, &n+1\notin I,\\
          0,&i\notin I, &n+1\notin I,\\
          0,&i\in I, &n+1\in I,\\
          1,&i\notin I, &n+1\in I.
        \end{array}\right.
    $$
    Consider a $\Z_2$-matrix $\lambda'$ of size $(n+1)\times (|\B|-1)$ which is defined by
    $$
    \lambda'_{iI}=\left\{
        \begin{array}{ll}
          1, & i\in I,\\
          0, & \textrm{otherwise}.
        \end{array}\right.
    $$
    It is trivial that the $i$-th row of $\lambda$ is the sum of the $i$-th and the $(n+1)$-th rows of $\lambda'$. In general, $\lambda_S$ is the sum of the $j$-th rows of $\lambda'$ over all $j\in T$, where $T=T(S)\subseteq [n+1]$ is
    $$
    T=\left\{
        \begin{array}{ll}
          S, & \textrm{if } |S| \textrm{ is even},\\
          S\cup\{n+1\}, &\textrm{if } |S| \textrm{ is odd}.
        \end{array}\right.
    $$
    The map $T$ is a one-to-one correspondence from the set of subsets of $[n]$ to the set of subsets of $[n+1]$ with even cardinality. Therefore, we have a new formula for our case in place of Theorem \ref{formula}:
    \begin{lemma}\label{newformula}
        Let $\B$ be a building set on $[n+1]$ and $M_\R(\B)$ its associated real toric manifold. Then the Betti number of $M_\R(\B)$ is given by
        $$\beta_i (M_\R(\B)) = \sum_{\atopp{T\subseteq [n+1]}{|T|=\textrm{even}}} \rank_\Q \tilde H_{i-1}(P'_T;\Q),
        $$
        where $P'_T$ is the union of every facet $F_I$ such that $|T\cap I|$ is odd.
    \end{lemma}

    \begin{example}
        Let $\B = \{1,2,3,12,23,123\}$. Then $\lambda$ is a $2\times 5$ matrix
        $$ \lambda =
            \begin{pmatrix}
                1 & 2 & 3 & 12 & 23 \\
                \hline
                1 & 0 & 1 & 1 & 1 \\
                0 & 1 & 1 & 1 & 0
            \end{pmatrix}.
        $$
        The zeroth row above the horizontal line was inserted only to indicate indexing of facets. For example, if $S=\{2\}$, then $P_S = F_2 \cup F_3 \cup F_{12}$. For $S=\{1,2\}$, the sum of the first and the second row is $(1 1 0 0 1)$ and therefore $P_S = F_1 \cup F_2 \cup F_{23}$.

        Next, $\lambda'$ is a $3 \times 5$ matrix
        $$ \lambda' =
            \begin{pmatrix}
                1 & 2 & 3 & 12 & 23 \\
                \hline
                1 & 0 & 0 & 1 & 0 \\
                0 & 1 & 0 & 1 & 1 \\
                0 & 0 & 1 & 0 & 1
            \end{pmatrix}.
        $$
        Note that $T$ always has even cardinality. If $S=\{2\}$, then $T=\{2,3\}$, and $P'_T=F_2 \cup F_3 \cup F_{12}$. In the case $S=\{1,2\}$, then $T=\{1,2\}$, and $P'_T=F_1 \cup F_2 \cup F_{23}$. Observe that $P_S=P'_T$.
    \end{example}

    Note that $P'_T\subseteq \partial P$ has its dual simplicial complex which is denoted by $K'_T$. Obviously $K'_T$ is an induced subcomplex of $\Delta_\B$. If there is no danger of confusion, we will abuse the notation $I$ for a facet $F_I$, its index set $I\in\B\setminus[n+1]$ or the corresponding vertex of $K'_T$.

From now on, unless otherwise noted, the building set $\B$ will be graphical with a connected graph $G$.

    \begin{definition}
    Let $\B=\B(G)$, where $G$ is a connected graph. We say that facets $\{F_{I_1},\ldots,F_{I_k}\}$ of $P_\B$ \emph{meet by inclusion} if there is a reindexing such that $I_1\subset\ldots\subset I_k$. We also say that facets $F_I$ and $F_J$ \emph{meet by separation} if $G|_{I\cup J}$ is a disconnected graph. In both cases we say $F_I$ and $F_J$ meet.
    \end{definition}

    We note that $\Delta_\B$ is a nested set complex by Theorem~\ref{nested}. Thus, if $F_I \cap F_J \neq \varnothing$, then $F_I$ and $F_j$ meet either inclusion or separation. Otherwise, that is, $F_I \cap F_J = \varnothing$, then
 $G|_{I\cup J}$ is connected and neither $I\subseteq J$ nor $J \subseteq I$. In this case we say $F_I$ and $F_J$ does not meet. In this paper, we will use the term `meet' in only this sense to avoid confusion. Thus, for example, if the distinct facets $I$ and $J$ meet, then $I\subset J$, $I\supset J$, or $I\cup J\notin \B$. If $I$ and $J$ does not meet, then $G|_{I\cup J}$ is connected.

    Before proceeding to general $T$, we first consider the situation when $G$ is a connected graph with $n+1=2k$ vertices and $T=[n+1]$ is the entire set. Assume that $n+1\geq 4$ to prevent trivial cases. Denote $P'_T$ by $P_G^{\mathrm{odd}}$  and denote $K'_T$ by $K_G^{\mathrm{odd}}$. Notice that $P_G^{\mathrm{odd}}$ is the union of every facet $F_I$ such that $|I|$ is odd. Similarly define $P_G^{\mathrm{even}}$ be the union of every facet $F_I$ such that $|I|$ is even. Its dual complex $K_G^{\mathrm{even}}$ is the induced subcomplex of $\Delta_{\B(G)}$ whose vertices have even cardinality. Their union $P_G^{\mathrm{odd}}\cup P_G^{\mathrm{even}} = \partial P_{\B(G)}$ is homeomorphic to the sphere $S^{n-1}$. Note that we are enough to compute $H_*(P_G^{\mathrm{even}},\Q)$ instead of $H_*(P_G^{\mathrm{odd}},\Q)$ by Alexander duality.

    \begin{lemma}\cite[Corollary 7.2]{PRW} \label{lem:flag}
        For a connected finite graph $G$, the simplicial complex $\Delta_{\B(G)}$ is a flag complex, i.e. $\Delta_{\B(G)}$ contains every clique of its 1-skeleton. Therefore, if $G$ has an even number of vertices, the simplicial complex $K_G^{\mathrm{even}}$ is also flag.
    \end{lemma}
    \begin{proof}
        Let $C=\{J_1,\ldots,J_\ell\}$ be a clique, i.e. any two of $F_{J_i}$'s meet by inclusion or meet by separation. Then $C$ satisfies (N\ref{n1}). To check (N\ref{n2}), we can assume that $J_i$'s are mutually disjoint. Any two of $F_{J_i}$'s meet by separation and that means any $G|_{J_i}$ cannot have outgoing edges in $G|_{J_1\cup\cdots\cup J_\ell}$. Thus $G|_{J_1\cup\cdots\cup J_\ell} = G|_{J_1} \cup \cdots \cup G|_{J_\ell}$ and it is disconnected.

       Since an induced subcomplex of a flag complex is flag, $K_G^{\mathrm{even}}$ is flag.
    \end{proof}

   \begin{definition}[\cite{PRW}]
        A simplicial complex $\Delta'$ is a \emph{geometric subdivision} of a simplicial complex $\Delta$ if they have geometric realizations that are topological spaces on the same underlying set, and every face of $\Delta'$ is contained in a single face of $\Delta$.
    \end{definition}
    \begin{lemma}\label{lem:subdivision1}
        Let $\B$ be a connected building set on $[n+1]$. Let $L$ be the order complex of the poset of nonempty proper subsets of $[n+1]$. Then $L$ is a geometric subdivision of the nested set complex $\Delta_\B$, where the face of $L$ corresponding to the chain $\varnothing \subsetneq I_1 \subsetneq \dotsb \subsetneq I_s \subsetneq [n+1]$ maps into the face of $\Delta_\B$ corresponding to the nested set consisting of all maximal elements of $\B|_{I_j}$ as $j$ runs over $1,\dotsc,s$.
    \end{lemma}
    \begin{proof}
        Proposition~3.2 of \cite{PRW} implies that every nested set complex $\Delta_\B$ can be geometrically subdivided to $\Delta_{\B(K_{n+1})}$, where $K_{n+1}$ is a complete graph. But $L$ is exactly the nested set complex $\Delta_{\B(K_{n+1})}$.
    \end{proof}

    \begin{lemma}\label{subd}
        Assume $G$ has even order. Let $\widehat{S_G}$ be the set of subsets of $V(G)$ such that for each element $I$ of $\widehat{S_G}$, the induced subgraph $G|_I$ has no connected components of odd order. Let $S_G = \widehat{S_G}\setminus\{\varnothing,V(G)\}$. Let $L_G^{\mathrm{even}}$ be the order complex of the poset $S_G$, that is, $L_G^{\mathrm{even}}$ is the simplicial complex whose faces are finite chains of $S_G$. Then $L_G^{\mathrm{even}}$ is a geometric subdivision of $K_G^{\mathrm{even}}$.
    \end{lemma}

    \begin{proof}
        The simplicial complex $K_G^{\mathrm{even}}$ is an induced subcomplex of $\Delta_{\B(G)}$ and $\Delta_{\B(G)}$ is subdivided to $L$ of Lemma~\ref{lem:subdivision1}. Observe that the corresponding subcomplex of $L$ is exactly $L_G^{\mathrm{even}}$.
    \end{proof}

    Keep in mind our objective is to compute the rational homology of $P_G^{\mathrm{even}}$ (actually $K_G^{\mathrm{even}}$). A simplicial complex is \emph{pure} if its every maximal simplex has the same dimension. A finite, pure simplicial complex $K$ of dimension $n$ is called \emph{shellable} if there is an ordering $C_1, C_2, \ldots$ of maximal simplices of $K$, called a \emph{shelling}, such that $(\bigcup_{i=1}^{k-1}C_i)\cap C_k$ is pure of dimension $n-1$ for every $k$. It is well known (\cite{S}) that shellable complexes are Cohen-Macaulay and thus homotopy equivalent to a wedge sum of spheres of the same dimension. In \cite{B}, Bj\"orner presented a criterion for shellability of order complexes. Let us introduce some notions and properties about posets. A poset is \emph{bounded} if it has a maximum and a minimum. Let $t$ and $s$ be elements of a poset. $t$ \emph{covers} $s$, denoted by $ t\gtrdot s$ or $s\lessdot t$, if $s<t$ and there is no $r$ such that $s<r<t$. A poset $S$ is \emph{graded} if there is an order-preserving function $\rho: S\to\N$, called a rank function, such that $\rho(t)=\rho(s)+1$ if $s\lessdot t$. A finite poset is called \emph{semimodular} if whenever two distinct elements $u$, $v$ both cover $t$ there is a $z$ which covers each of $u$ and $v$. A poset is said to be \emph{locally semimodular} when all intervals $[a,b]=\{x\mid a\le x\le b\}$ are semimodular.
    \begin{theorem}\cite[Theorem 6.1]{B}\label{thm:bjorner}
        Suppose that a finite poset is bounded and locally semimodular. Then its order complex is shellable.
    \end{theorem}
%    \begin{remark}
%        In our case, the boundedness condition is actually not essential for shellability. For a finite poset $S$, define $\widehat{S}:=S \cup \{\hat{0}, \hat{1}\}$ and extend the partial order so that $\hat{0}< x< \hat{1}$ for any $x\in S$. Then $\widehat{S}$ is a bounded poset.
%    \end{remark}
    \begin{proposition}\label{shell}
        The poset $\widehat{S_G}$ is bounded and locally semimodular. Hence, the simplicial complex $L_G^{\mathrm{even}}$ is shellable of dimension $k-2$.
    \end{proposition}
    \begin{proof}
        When $J\subset I$ and $G|_J$ is a component of $G|_I$, let us call $J$ simply a component of $I$. First, note that $\widehat{S_G}$ is a graded poset with rank function $\rho(I)= \frac{|I|}2$. Suppose that $[a,b]$ is an interval in $\widehat{S_G}$ and $t\in[a,b]$. Suppose that $a\le t\lessdot u\le b$, $a\le t\lessdot v\le b$, and $u\ne v$. Then $u<b$ and $v<b$ since $u$ and $v$ are distinct and $|u|=|v|$. Consider the set $u \cup v \subseteq b$. Be careful $u\cup v$ is not necessarily an element of $\widehat{S_G}$. There are two cases: $|u\cup v|=|u|+1$ and $|u\cup v|=|u|+2$. Note that $|u\cup v|\le|u|+2$ since $|u|=|v|=|t|+2$ and $t\subset u\cap v$.

Suppose the first case, i.e., $|u\cup v|=|u|+1$. Then $u\cup v = u \cup \{q\}$ for some $q\in v$. The set $u\cup v$ has a unique component of odd cardinality,  say $U$, which contains $q$. Since every component of $b$ has even cardinality and $U\subset b$, there is a set $\bar{U}\subset b$ containing $U$ and having cardinality $|U|+1$. Then the set $u\cup v \cup \bar{U}$ covers both $u$ and $v$ and is smaller than $b$. Beware that $\bar{U}$ need not be a component of $u\cup v \cup \bar{U}$.

On the other hand, suppose that $|u\cup v|=|u|+2$. Then $u=t\cup\{p,q\}$ and $v=t\cup\{r,s\}$, where $p,q,r$, and $s$ are all distinct elements of $V(G)$. Since every connected component of $u$ has even cardinality, $p$ and $q$ lie in the same component of $u$. The same applies for $r,s\in v$. Consider the set $u\cup v = t\cup \{p,q,r,s\}$. It is obvious that every component of $u\cup v$ has even cardinality, i.e., $u\cup v \in \widehat{S_G}$, therefore we are done.

In conclusion, the order complex of the poset $\widehat{S_G}$ is shellable by Theorem~\ref{thm:bjorner}. Since any facet of the order complex of $\widehat{S_G}$ contains the vertices $\varnothing$ and $V(G)$, $L_G^{\mathrm{even}}$ is also a shellable simplicial complex, reminding that $L_G^{\mathrm{even}}$ is the order complex of the poset $S_G$. It is pure of dimension $k-2$ since any maximal chain of $S_G$ is of length $k-1$.
    \end{proof}

    \begin{remark}
        While preparing the publication of this paper, the authors have realized that the signed $a$-number of $G$ can be described entirely by the poset $\widehat{S_G}$. Namely, let $G$ be a graph of vertex set $V$. Then $sa(G)$ is given by the M\"obius function
        \[
            sa(G) = \mu(\varnothing, V)
        \]
        of the poset $\widehat{S_G}$ (we need to add $V$ into $\widehat{S_G}$ if $|V|$ is odd). Furthermore, $a_i(G)$ is the $i$-th Whitney number of the first kind of $\widehat{S_G}$.
    \end{remark}
    
    \begin{corollary}\label{pg}
        Let $G$ be a connected graph of $2k$ vertices. Then the integral homology of $P_G^{\mathrm{odd}}$ is:
        $$ \widetilde{H}_i(P_G^{\mathrm{odd}})=\left\{\begin{array}{ll}
        \Z^{a}, & i=k-1, \\
        0,      & \text{otherwise,} \\
        \end{array}\right.
        $$
        where $a=\ta(G)$ is an integer determined by the graph $G$ only. We temporarily call $\ta(G)$ the \emph{topological $a$-number} of $G$.
%        $P_G$ is homotopy equivalent to a wedge sum of $(k-1)$-spheres:
%        $$ P_G^{\mathrm{odd}} \simeq \bigvee^{a} S^{k-1}. $$
    \end{corollary}
    \begin{proof}
        By Proposition \ref{shell}, $L_G^{\mathrm{even}}$ is homotopy equivalent to a wedge sum of $(k-2)$-dimensional spheres. Lemma~\ref{subd} and Alexander duality (see \cite[Theorem~3.44]{H} for reference) imply the expected result.
    \end{proof}

\begin{example}
  Recall the settings as in Example~\ref{example:associahedron}. Then, $P_G^{\mathrm{even}}$ is the union of $3$ disjoint facets, and $P_G^{\mathrm{odd}}$ its complement on $\partial P_{\B(G)}$ which is homeomorphic to the $3$-punctured sphere which is homotopy equivalent to the wedge sum $S^1 \vee S^1$ of two circles. Hence, $\ta(P_4)=2$. See Figure~\ref{fig:pgeven}.
    \begin{figure}[t]
    \begin{center}
\begin{tikzpicture}[thick, scale=0.5]
    \fill[color=green] (5.3,1)--(6.3,2.5)--(2.3,6)--(1.3,4.5)--cycle
        (.6,6.3)--(-.7,5.5)--(-.7,-.5)--(.8,0)--cycle
        (5.3,-.5)--(5.7,-2.3)--(-3.5,-6)--(-4.3,-4.3)--cycle;
    \draw (4,3.5) node {\Large{\red{$23$}}}
        (0.5,-3.5) node {\Large{\red{$12$}}}
        (0.2,2.2) node {\Large{\red{$34$}}};
    \draw (5.3,1)--(5.3,-0.5)--(5.7,-2.3)--(7.2,0)--(6.3,2.5)--cycle
        (1.3,4.5)--(2.3,6)--(0.6,6.3)--(-0.7,5.5)--(-1.7,4)--cycle;
    \draw (6.3,2.5)--(2.3,6)
        (5.3,1)--(1.3,4.5)
        (5.3,-.5)--(-4.3,-4.3)
        (5.7,-2.3)--(-3.5,-6)
        (-1.7,4)--(-4.3,-4.3)
        (-3.5,-6)--(-4.3,-4.3);
    \draw[dotted] (-3.5,-6)--(-0.7,-0.5)
        (0.8,0)--(-0.7,-0.5)
        (0.6,6.3)--(0.8,0)
        (-0.7,5.5)--(-0.7,-0.5)
        (0.8,0)--(7.2,0);
\end{tikzpicture}
    \end{center}
    \caption{The set $P_G^{\mathrm{even}}$ indicates three holes in a sphere.}
    \label{fig:pgeven}
    \end{figure}
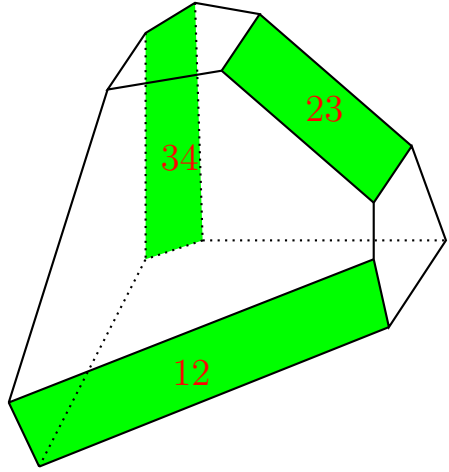
\end{example}

\section{Rational Betti numbers of $M(G)$}\label{sec:betti}
    In this section, we compute the rational homology of $P'_T$ for general $T$.% For example, consider the case $|T|=2$ first.
%    \begin{proposition}
%        Let $T$ be a subset of $[n+1]$, $n\geq 2$, with cardinality $2$. If $T$ represents an edge in $G$, then $P'_T$ is homotopy equivalent to a 0-sphere $S^0$. Otherwise, $P'_T$ is contractible.
%    \end{proposition}
%    \begin{proof}
%        Suppose $T=\{i,j\}$ is an edge. Then for every facet $F_I$ of $P'_T$, $I$ contains either $i$ or $j$, and not both of them. Observe the facets $F_{\{i\}}$ and $F_{\{j\}}$ are disjoint since $\{i,j\}$ is an element of $\B$. Thus $K'_T$ is a disjoint union of two cones whose apexes are $\{i\}$ and $\{j\}$ respectively.
%
%        Suppose $T$ is not an edge. Then $F_{\{i\}}$ and $F_{\{j\}}$ meet by separation and $K'_T$ is contractible.
%    \end{proof}
%    Actually, this is a special case of the following:
    \begin{proposition}\label{pt}
        Let $G$ be a connected graph on $[n+1]$ and $T\subseteq [n+1]$ be a subset with cardinality $2k$. Suppose $G|_T$ has $\ell$ components, $G_1,\ldots,G_\ell$. If some component of $G|_T$ has an odd number of vertices, then $P'_T$ is contractible, and hence, $\rank_\Q \widetilde{H}_i(P'_T;\Q)=0$ for all $i$. Otherwise, that is, if each component has even order, then,
        $$ \rank_\Q \widetilde{H}_i(P'_T;\Q)=\left\{\begin{array}{ll}
        \ta(G_1)\cdot \cdots \cdot \ta(G_\ell), & i=k-1, \\
        0,      &  i\neq k-1 \\
        \end{array}\right.
        $$
        where $\ta(G_i)$ is the topological $a$-number of $G_i$.
    \end{proposition}
    We extend the notion of topological $a$-numbers to general graphs. Note that we have already defined topological $a$-numbers for connected graphs with even order. Everything goes the same as its combinatorial sibling $a(G)$. Let $G$ be a finite graph. Then $\ta(G)=0$ if $G$ has a component of odd order. Otherwise, $\ta(G)$ is defined as the product of topological $a$-numbers of each component of $G$. As a convention, we define $ta(\varnothing)=1$ for the empty graph $\varnothing$.

    We introduce some lemmas to prove Proposition~\ref{pt}.
    \begin{lemma}\label{elimv}
        Let $p$ be a vertex of a simplicial complex $\Delta$ and suppose that the link of $p$, $\operatorname{Lk}p$, is contractible. Then $\Delta$ is homotopy equivalent to the complex $\Delta':=\Delta\setminus\operatorname{St}p$, where $\operatorname{St}p$ is the star of $p$.
    \end{lemma}
    \begin{proof}
        Observe that the closure of $\operatorname{St}p$ is the cone over $\operatorname{Lk}p$ with apex $p$. By gluing $(\operatorname{Lk}p)\times I$, $I=[0,1]$, to $\Delta'$ along $\operatorname{Lk}p$ by identifying $(\operatorname{Lk}p)\times \{0\}=\operatorname{Lk}p$, we obtain a new space $\Delta''$. Note $\Delta''/((\operatorname{Lk}p)\times \{1\}) = \Delta$. But $(\Delta'',(\operatorname{Lk}p)\times \{1\})$ is a CW pair. Thus by \cite[Proposition 0.17]{H}, $\Delta''$ is homotopy equivalent to $\Delta$. It is obvious that $\Delta''\simeq \Delta'$ since one have the natural deformation retraction.
    \end{proof}

    \begin{lemma}\label{subgraph}
        Let $T$ be a subset of $[n+1]$, $n\geq 2$, with even cardinality. Denote by $P''_T$ the union of facets $F_I$ such that $I\subseteq T$ and $|I|$ is odd. Then $P'_T$ is homotopy equivalent to $P''_T$.
%    Suppose $G_1, G_2, \ldots, G_\ell$ are the connected components of $G|_T$.
%    Then $P'_T$ is either contractible or homotopy equivalent to the space $P_{G|_T}$.
    \end{lemma}
    Remember that $P'_T$ is the union of every facet $F_I$ such that $|T\cap I|$ is odd. Thus, $P''_T\subseteq P'_T\subseteq\partial P$. We use the notation $K''_T$ for the dual complex of $P''_T$.
    \begin{proof}
        Let $F_I\subset P'_T$ be a facet in $P'_T$ and $I\in K'_T$ be the corresponding vertex. $I$ can be uniquely written as $I=J\amalg X$, where $J\subseteq T$ and $X\subseteq [n+1]\setminus T$. Be careful that $J$ is not necessarily a facet, but its cardinality is surely odd. Define by $|J|$ the $j$-degree of $I$ and by $|X|$ the $x$-degree of $I$. By definition, $P''_T$ is the union of facets in $P'_T$ whose $x$-degree is zero.

        By induction on $j$-degrees and $x$-degrees of $I$, we are going to eliminate all facets of nonzero $x$-degrees using Lemma \ref{elimv}. Consider $\operatorname{Lk}I\subset K'_T$. Since our complex is flag by Lemma~\ref{lem:flag}, $\operatorname{Lk}I$ is induced by its vertices, which `meet' $I$. Pick a vertex $L$ of $\operatorname{Lk}I$ other than $I$. Then $L$ meets $I$ and thus $L$ is included in $I$, includes $I$, or meets $I$ by separation. Since $J$ has odd cardinality, $G|_J$ has a component of odd order, say $G|_{J_1}$. If $L$ includes $I$ or meets $I$ by separation, then $L$ meets $J_1$. If $L$ is included in $I$ and the $x$-degree of $L$ is zero, then $L$ is included in $J_1$ or meets $J_1$ by separation, and also meets $J_1$. The remaining case is that $L\subsetneq I$ and $x$-degree of $L$ is nonzero. But then the $j$-degree of $L$ is lesser than that of $I$ and $L$ would have been already removed at some previous stage of the induction.

        In conclusion, $\operatorname{Lk}I$ is a cone with apex $J_1$, therefore it is contractible. By Lemma \ref{elimv}, we can `delete' the vertex $I$ without changing the homotopy type of the simplicial complex $K'_T$ until remaining complex is $K''_T$.
    \end{proof}

    Now we can prove Proposition~\ref{pt}.

    \begin{proof}[proof of Proposition~\ref{pt}]
        We use Lemma \ref{subgraph} to compute the homology. First, we deal with the case every component of $G|_T$ has an even number of vertices. Assume that $G|_T=G_1\amalg\cdots\amalg G_\ell$ and $\left|V(G|_T)\right|=2k$ and $|V(G_i)|=2k_i$, therefore $k_1+\cdots +k_\ell = k$. Recall that the simplicial join of two simplicial complexes $\Delta_1$ and $\Delta_2$ is the simplicial complex $\Delta_1\star\Delta_2$ whose simplex is given by $\{v_0,\ldots, v_p,w_0,\ldots,w_q\}$ if $\{v_0,\ldots,v_p\}$ and $\{w_0,\ldots,w_q\}$ are simplices of $\Delta_1$ and $\Delta_2$ respectively. If $\ell=1$, then $K''_T=K_G^{\mathrm{odd}}$. If $\ell\ge 2$, observe that the simplicial complex $K''_T$ is the simplicial join of $K_{G_i}^{\mathrm{odd}}$'s. Denote by $n_i+1=2k_i$ the number of vertices of $G_i$. The join of $A$ and $B$, $A\star B$, is homotopy equivalent to the (reduced) suspension of the smash product of $A$ and $B$, i.e., $A\star B\simeq \Sigma(A\wedge B)=S^1\wedge A\wedge B$. We have a reduced version of the Ku\"nneth formula (see \cite[page 223]{H} for a reference)
        $$
            \widetilde{H}_*(A\wedge B;\Q)\cong \widetilde{H}_*(A;\Q)\otimes_\Q\widetilde{H}_*(B;\Q).
        $$

         Note that the homology of $A\star B$ is determined by $H_*(A)$ and $H_*(B)$. By Corollary \ref{pg}, $K_{G_i}^{\mathrm{odd}}$ has the same homology as that of the wedge sum
        $$\bigvee^{\ta(G_i)}_{j=1} S^{k_i-1}$$
        and $\widetilde{H}_{k_i-1}(K_{G_i}^{\mathrm{odd}};\Q)=\Q^{\ta(G_i)}$.
        Thus the join is computed like the following
        $$
            K_{G_1}^{\mathrm{odd}}\star\cdots\star K_{G_\ell}^{\mathrm{odd}}\simeq \underbrace{S^1\wedge\cdots\wedge S^1}_{(\ell-1) \text{ times}}\wedge K_{G_1}^{\mathrm{odd}}\wedge\cdots\wedge K_{G_\ell}^{\mathrm{odd}}=S^{\ell-1} \wedge \bigwedge_{i=1}^\ell K_{G_i}^{\mathrm{odd}}
        $$

        and its homology is
        $$
            \widetilde{H}_{k-1}(K''_T;\Q)=\widetilde{H}_{\ell-1}(S^{\ell-1};\Q)\otimes \bigotimes_{i=1}^\ell \widetilde{H}_{k_i-1}(K_{G_i}^{\mathrm{odd}};\Q)=\Q^{\ta(G_1)\cdots \ta(G_\ell)}
        $$

%        $$ \bigvee^{a(G_1)}_{j=1} S^{k_1-1}\star\ldots\star\bigvee^{a(G_\ell)}_{j=1} S^{k_\ell-1}\simeq \bigvee^{\prod a(G_i)} S^{k-1},$$
        since $\ell-1+(k_1-1)+\cdots+(k_\ell-1) = \sum k_i -1 = k-1$.

        On the other hand, suppose there is a component, say $G_1=G|_{I_1}$, of $G|_T$ of odd order. Then $I_1$ is a vertex of $K''_T$. Moreover, $I_1$ meets every other vertex of $K''_T$. Hence $K''_T$ is contractible.
    \end{proof}

    Now, we are ready to prove Theorem~\ref{thm:mainintro}.

%    \begin{theorem}\label{thm:mainintro}
%        Let $G$ be a connected graph. Then the Betti number of $M(G)$ is given by
%        $$
%            \beta_i(M(G))=\sum_{\atopp{I\subseteq V(G)}{|I|=2i}} a(G|_I).
%        $$
%        In particular, $\beta_1(M)$ equals the number of edges of $G$.
%    \end{theorem}
\begin{proof}[Proof of Theorem~\ref{thm:mainintro}]
 Define the \emph{topological signed $a$-number} of $G$, denoted by $\tsa(G)$, as follows:
    $$
    \tsa(G)=\left\{
        \begin{array}{ll}
          (-1)^k \ta(G), & \text{if $G$ has $2k$ vertices, } k\ge 0,\\
          0, & \text{otherwise}.
        \end{array}\right.
    $$

    Assume $G$ is connected.
    By combining Lemma~\ref{newformula} and Proposition~\ref{pt}, we have that
        $$
            \beta_i(M(G))=\sum_{\atopp{I\subseteq V(G)}{|I|=2i}} \ta(G|_I).
        $$

    If $G$ has odd order, then $\tsa(G)=0$ by definition. Suppose that $G$ has even order. Then the dimension of $M(G)$ is odd and its Euler characteristic is zero, therefore we obtain the formula \eqref{rec} for topological $a$-numbers. This result matches the original $a$-numbers with the topological ones, proving they are the same graph invariants. In other words, $a(G)=ta(G)$ and $sa(G)=tsa(G)$.

    Now, we assume that $G$ is not connected. Let $G=G_1 \amalg \cdots \amalg G_\ell$. Then, $P_{\B(G)} = P_{\B(G_1)} \times \cdots \times P_{\B(G_\ell)}$, and $M(G) = M(G_1) \times \cdots \times M(G_\ell)$. Therefore,
\begin{align*}
    \beta_i(M(G)) &= \sum_{j_1+\cdots+j_\ell=i} \beta_{j_1}(M(G_1))\cdot \cdots \cdot \beta_{j_\ell}(M(G_\ell))\\
    &= \sum_{j_1 + \cdots + j_\ell=i} a_{j_1}(G_1)\cdot \cdots \cdot a_{j_\ell}(G_\ell) \\
    &= \sum_{j_1 + \cdots + j_\ell=i} \prod_{k=1}^\ell \sum_{\atopp{I \subset V(G_k)}{|I|=2j_k}} a(G_k|_I)\\
    &= \sum_{\atopp{I \subset V(G)}{|I|=2i}} \prod_{k=1}^\ell a(G_k|_I)\\
    &= \sum_{\atopp{I \subset V(G)}{|I|=2i}} a\left(\left.\coprod_{k=1}^\ell G_k \right|_I \right)\\
    &= a_i(G),
\end{align*} which proves the theorem.
\end{proof}

\bigskip

\section*{Acknowledgements}
The authors would like to appreciate Jang Soo Kim and Heesung Shin for suggesting combinatorial proofs of Theorems in Section~\ref{sec:anumber} and to Kyoung Sook Park for useful discussion on earlier version of this paper. They are also thankful to the anonymous reviewer of this paper who gave several valuable comments including introduction of Lemma~\ref{lem:subdivision1} and a new proof of Lemma~\ref{subd} which has replaced the original one.

\end{document}